\documentclass[11pt, amsfonts, dvipdfmx]{amsart}

\usepackage{amsmath,amssymb,amsthm}
\usepackage[colorlinks=true]{hyperref}
\usepackage[all]{xy}

\numberwithin{equation}{section}

\SelectTips{eu}{12}

\newtheorem{theorem}{Theorem}[section]
\newtheorem{proposition}[theorem]{Proposition}
\newtheorem{lemma}[theorem]{Lemma}
\newtheorem{corollary}[theorem]{Corollary}

\theoremstyle{definition}
\newtheorem{definition}[theorem]{Definition}

\newtheorem{problem}[theorem]{Problem}

\theoremstyle{remark}
\newtheorem{remark}[theorem]{Remark}

\newcommand{\Z}{\mathbb{Z}}

\newcommand{\R}{\mathbb{R}}

\newcommand{\F}{\mathbb{F}}

\newcommand{\Tor}{\operatorname{Tor}}
\newcommand{\lk}{\operatorname{lk}}

\title{Golod and tight 3-manifolds}

\author{Kouyemon Iriye}
\address{Department of Mathematical Sciences, Osaka Prefecture University, Sakai, 599-8531, Japan}
\email{kiriye@mi.s.osakafu-u.ac.jp}

\author{Daisuke Kishimoto}
\address{Department of Mathematics, Kyoto University, Kyoto, 606-8502, Japan}
\email{kishi@math.kyoto-u.ac.jp}



\begin{document}

  \maketitle

  \baselineskip.525cm

  \begin{abstract}
    The notions Golodness and tightness for simplicial complexes come from algebra and geometry, respectively. We prove these two notions are equivalent for 3-manifold triangulations, through a topological characterization of a polyhedral product for a tight-neighborly manifold triangulation of dimension $\ge 3$.
  \end{abstract}


  \section{Introduction}\label{introduction}

  Let $\F$ be a field, and let $S=\F[x_1,\ldots,x_m]$, where we assume each $x_i$ is of degree 2. Serre \cite{S} proved that for $R=S/I$ where $I$ is a homogeneous ideal of $S$, there is a coefficientwise inequality
  \[
    P(\Tor^R(\F,\F);t)\le\frac{(1+t^2)^m}{1-t(P(\Tor^S(R,\F);t)-1)}
  \]
  where $P(V;t)$ denotes the Poincar\'e series of a graded vector space $V$. In the extreme case that the equality holds, $R$ is called \emph{Golod}. It was Golod who proved that $R$ is Golod if and only if all products and (higher) Massey products in the Koszul homology of $R$ vanish, where the Koszul homology of $R$ is isomorphic with $\Tor^S(R,\F)$ as a vector space.

  Let $K$ be a simplicial complex with vertex set $[m]=\{1,2,\ldots,m\}$. Let $\F[K]$ denote the Stanley-Reisner ring of $K$ over $\F$, where we assume generators of $\F[K]$ are of degree 2. Then $\F[K]$ expresses combinatorial properties of $K$, and conversely, it is of particular interest to translate a given algebraic property of the Stanley-Reisner ring $\F[K]$ into a combinatorial property of $K$. We say that $K$ is \emph{$\F$-Golod} if $\F[K]$ is Golod. We aim to characterize Golod complexes combinatorially.

  Recently, a new approach to a combinatorial characterization of Golod complexes is taken. We can construct a space $Z_K$, called the \emph{moment-angle complex} for $K$, in accordance with the combinatorial information of $K$. Then combinatorial properties are encoded on the topology of $Z_K$, and in particular, Golodness can be read from a homotopical property of $Z_K$ as follows. Baskakov, Buchstaber and Panov \cite{BBP} proved that the cohomology of $Z_K$ with coefficients in $\F$ is isomorphic with the Koszul homology of $\F[K]$, where the isomorphism respects products and (higher) Massey products. Then it follows that $K$ is Golod over any field whenever $Z_K$ is a suspension, and so Golod complexes have been studied also in connection with desuspension of $Z_K$ and a more general \emph{polyhedral product} \cite{GPTW,GT,GW,IK0,IK1,IK2,IK3,IK4,IK5}. See a survey \cite{BBC} for more information about moment-angle complexes and polyhedral products. Here we remark that there is a Golod complex $K$ such that $Z_K$ is not a suspension as shown by Yano and the first author \cite{IY}.

  In \cite{IK1,IK3,IK4}, the authors characterized Golod complexes of dimension one and two in terms of both combinatorial properties of $K$ and desuspension of $Z_K$. Here we recall a characterization  of Golodness of a closed connected surface triangulation, proved in \cite{IK1}. Recall that a simplicial complex is called \emph{neighborly} if every pair of vertices forms an edge.

  \begin{theorem}
    \label{surface}
    Let $S$ be a triangulation of a closed connected $\F$-orientable surface. Then the following statements are equivalent:
    \begin{enumerate}
      \item $S$ is $\F$-Golod;

      \item $S$ is neighborly;

      \item $Z_S$ is a suspension.
    \end{enumerate}
  \end{theorem}

  We introduce another notion of a simplicial complex coming from geometry. S.-S. Chern and R.K. Lashof proved that the total absolute curvature of an immersion $f\colon M\to\R^n$ of a compact manifold $M$ is bounded below by the Morse number of some Morse function on $M$. On the other hand, the Morse number is bounded below by the Betti number. Tightness of an immersion $f$ is defined by the equality between the total absolute curvature of an immersion $f$ and the Betti number of $M$, which is the case that the total absolute curvature is minimal and attained by the Betti number. See \cite{KL,Kui}. It is known that an immersion $f$ is tight if and only if for almost every closed half-space $H$, the inclusion $f(M)\cap H\to f(M)$ is injective in homology.

  Tightness of a simplicial complex is defined as a combinatorial analog of tightness of an immersion. See \cite{KL} for details. Let $K$ be a simplicial complex with vertex set $[m]$. For $\emptyset\ne I\subset[m]$, the full subcomplex of $K$ over $I$ is defined by
  \[
    K_I=\{\sigma\in K\mid\sigma\subset I\}.
  \]

  \begin{definition}
    Let $K$ be a connected simplicial complex with vertex set $[m]$. We say that $K$ is \emph{$\F$-tight} if the natural map $H_*(K_I;\F)\to H_*(K;\F)$ is injective for each $\emptyset\ne I\subset[m]$.
  \end{definition}

  Golodness and tightness have origins in different fields of mathematics, algebra and geometry, respectively. The aim of this paper is to prove the seemingly irrelevant these two notions are equivalent for 3-manifold triangulations through the topology of $Z_K$ or a more general \emph{polyhedral product} (see Section \ref{polyhedral product}). Now we state the main theorem.

  \begin{theorem}
    \label{main}
    Let $M$ be a triangulation of a closed connected $\F$-orientable 3-manifold. Then the following statements are equivalent:
    \begin{enumerate}
      \item $M$ is $\F$-Golod;

      \item $M$ is $\F$-tight;

      \item $Z_M$ is a suspension.
    \end{enumerate}
  \end{theorem}

  Recall that a $d$-manifold triangulation is called \emph{stacked} if it is the boundary of a $(d+1)$-manifold triangulation whose interior simplices are of dimension $\ge d$. Stacked manifold triangulations have been studied in several directions, and we will use its connection to tightness (Section \ref{tightness}). See \cite{BDS,DM,KL} and references therein for more on stacked manifold triangulations. Bagchi, Datta and Spreer \cite{BDS} (cf. Theorem \ref{tight vs tight-neighborly}) proved that a closed connected $\F$-orientable 3-manifold triangulation is $\F$-tight if and only if it is neighborly and stacked. Then we get the following corollary of Theorem \ref{main}, which enables us to compare with the 2-dimensional case Theorem \ref{surface}.

  \begin{corollary}
    \label{main corollary}
    Let $M$ be a triangulation of a closed connected $\F$-orientable 3-manifold. Then the following statements are equivalent:
    \begin{enumerate}
      \item $M$ is $\F$-Golod;

      \item $M$ is neighborly and stacked;

      \item $Z_M$ is a suspension.
    \end{enumerate}
  \end{corollary}

  We will investigate a relation between Golodness and tightness of $d$-manifold triangulations for $d\ge 3$, not only for $d=3$, through tight-neighborliness. We will prove the following theorem, where Theorem \ref{main} is its special case $d=3$.

  \begin{theorem}
    \label{main general}
    Let $M$ be a triangulation of a closed connected $\F$-orientable $d$-manifold for $d\ge 3$, and consider the following conditions:
    \begin{enumerate}
      \item $M$ is $\F$-Golod;

      \item $M$ is $\F$-tight;

      \item $M$ is tight-neighborly;

      \item the fat-wedge filtration of $\R Z_M$ is trivial.
    \end{enumerate}
    Then there are implications
    \[
      (1)\quad\Longrightarrow\quad(2)\quad\Longleftarrow\quad(3)\quad\Longrightarrow\quad(4)\quad\Longrightarrow\quad(1).
    \]
    Moreover, for $d=3$, the implication $(2)\,\Rightarrow\,(3)$ also holds, so all conditions are equivalent.
  \end{theorem}

  Remarks on Theorem \ref{main general} are in order. Tight-neighborly triangulations of $d$-manifolds for $d\ge 3$ will be defined in Section \ref{tightness}. The space $\R Z_K$ is the real moment-angle complex, and properties of its fat-wedge filtration will be given in Section \ref{polyhedral product}. In particular, we will see that if the fat-wedge filtration of $\R Z_K$ is trivial, then $Z_K$ is a suspension. So Theorem \ref{main} is the special case of Theorem \ref{main general} for $d=3$ as mentioned above. Datta and Murai \cite{DM} proved that if $M$ is tight-neighborly and $d\ge 4$, then it is $\F$-tight and $\beta_i(M;\F)=0$ for $2\le i\le d-2$, where $\beta_i(M;\F)=\dim H_i(M;\F)$ denotes the $i$-th Betti number. So if $\beta_i(M;\F)=0$ for $2\le i\le d-2$ and $d\ge 4$, then all conditions in Theorem \ref{main general} are equivalent, where the triviality of the Betti numbers is necessary because there is a $\F$-tight 9-vertex triangulation of $\mathbb{C}P^2$ for any field $\F$ as in \cite[Example 3.15]{BD1} which is not tight-neighborly.

  The paper is organized as follows. Section \ref{tightness} collects properties of tight and tight-neighborly manifold triangulations that will be needed in later sections. Section \ref{weak Golodness} introduces a weak version of Golodness and proves the weak Golodness implies tightness of orientable manifold triangulations. Section \ref{F(M)} investigates a simplicial complex $F(M)$ constructed from a tight-neighborly $d$-manifold triangulation $M$ for $d\ge 3$, and Section \ref{polyhedral product} recalls the fat-wedge filtration technique for polyhedral products which is the main ingredient in desuspending $Z_K$. Section \ref{proof} applies the results in Sections \ref{F(M)} and \ref{polyhedral product} to prove Theorem \ref{main general}. Finally, Section \ref{problems} poses two problems for a further study of a relationship between Golodness and tightness.\\

  \textit{Acknowledgement:} The first author was supported by JSPS KAKENHI 26400094, and the second author was supported by JSPS KAKENHI 17K05248.


  \section{Tightness}\label{tightness}

  This section collects facts about tight and tight-neighborly manifold triangulations that we will use. As mentioned in Section \ref{introduction}, tightness of a simplicial complex is a discrete analog of a tight space studied in differential geometry with connection to minimality of the total absolute curvature, and tight complexes have been studied mainly for manifold triangulations. First, we show:

  \begin{lemma}
    \label{tight neighborly}
    Every $\F$-tight complex is neighborly.
  \end{lemma}

  \begin{proof}
    Let $K$ be an $\F$-tight complex. Then for two vertices $v,w$ of $K$, the natural map $H_0(K_{\{v,w\}};\F)\to H_0(K;\F)$ is injective. Since $K$ is connected, $H_0(K;\F)\cong\F$, and so $H_0(K_{\{v,w\}};\F)\cong\F$. Then $v$ and $w$ must be joined by an edge.
  \end{proof}

  Next, we explain a conjecture on tight manifold triangulations. Let $|K|$ denote the geometric realization of $K$, and let $f(K)=(f_0(K),f_1(K),\ldots,f_{\dim K}(K))$ denotes the $f$-vector of $K$. We say that $K$ is \emph{strongly minimal} if for any simplicial complex $L$ with $|K|\cong|L|$, it holds that
  \[
    f_i(K)\le f_i(L)
  \]
  for each $i\ge 0$. K\"{u}hnel and Lutz \cite{KL} conjectured that every $\F$-tight triangulation of a closed connected manifold is strongly minimal. Clearly, the only $\F$-tight closed connected 1-manifold triangulation is the boundary of a 2-simplex, so the conjecture is true in dimension 1. Moreover, the 2-dimensional case was verified as mentioned in \cite{KL}, and the 3-dimensional case was verified by Bagchi, Datta and Spreer \cite{BDS}. But the case of dimensions
 $\ge 4$ is still open.

  As for minimality of a manifold triangulation, we have another notion introduced by Lutz, Sulanke and Swartz \cite{LSS}.

  \begin{definition}
    A closed connected $d$-manifold triangulation $M$ with vertex set $[m]$ for $d\ge 3$ is \emph{tight-neighborly} if
    \[
      \binom{m-d-1}{2}=\binom{d+2}{2}\beta_1(M;\F).
    \]
  \end{definition}

  Tight-neighborly manifold triangulations are known to be vertex minimal. By definition, tight-neighborliness seems to depend on the ground field $\F$, but it is actually independent of the ground field $\F$ as tight-neighborly manifold triangulations are neighborly and stacked. Tightness and tight-neighborliness have the following relation. Let $S^1\widetilde{\times}S^{d-1}$ denote a non-trivial $S^{d-1}$-bundle over $S^1$.

  \begin{theorem}
    \label{tight vs tight-neighborly}
    Let $M$ be a closed connected $\F$-orientable $d$-manifold triangulation for $d\ge 3$, and consider the following conditions:

    \begin{enumerate}
      \item $M$ is $\F$-tight;

      \item $M$ is tight-neighborly;

      \item $M$ is neighborly and stacked;

      \item $M$ has the topological type of either
      \[
        S^d,\quad(S^1\widetilde{\times}S^{d-1})^{\# k},\quad(S^1\times S^{d-1})^{\# k}.
      \]
    \end{enumerate}
    Then there are implications
    \[
      (1)\quad\Longleftarrow\quad(2)\quad\Longleftrightarrow\quad(3)\quad\Longrightarrow\quad(4).
    \]
    Moreover, the implication $(1)\,\Rightarrow\,(2)$ also holds for $d=3$.
  \end{theorem}

  \begin{proof}
    The implications are shown in \cite{DM} for $d\ge 4$ and \cite{BDS} for $d=3$.
  \end{proof}

  \begin{remark}
    The integer $k$ in Theorem \ref{tight vs tight-neighborly} for $d=3$ is known to satisfy $80k+1$ is a perfect square. For $k=1,30,99,208,357,546$, tight-neighborly triangulations of $(S^1\widetilde{\times}S^2)^{\# k}$ are constructed in \cite{BDSS}, but no tight-neighborly triangulation of $(S^1\times S^2)^{\# k}$ is not known.
  \end{remark}


  \section{Weak Golodness}\label{weak Golodness}

  This section introduces weak Golodness and study it for manifold triangulations. Let $K$ be a simplicial complex with vertex set $[m]$, and let $\mathcal{H}_*(\F[K])$ denote the Koszul homology of the Stanley-Reisner ring $\F[K]$. As mentioned in Section \ref{introduction}, $K$ is $\F$-Golod if and only if all products and (higher) Massey products in $\mathcal{H}_*(\F[K])$ vanish. Now we define weak Golodness.

  \begin{definition}
    A simplicial complex $K$ is \emph{weakly $\F$-Golod} if all products in $\mathcal{H}_*(\F[K])$ vanish.
  \end{definition}

  Clearly, $K$ is weakly $\F$-Golod whenever it is $\F$-Golod. Berglund and  J\"{o}llenbeck \cite{BJ} stated that Golodness and weak Golodness of every simplicial complex are equivalent, but this was disproved by Katth\"{a}n \cite{Ka}. Then defining weak Golodness makes sense.

  We recall a combinatorial description of the multiplication of $\mathcal{H}_*(\F[K])$. For disjoint non-empty subsets $I,J\subset[m]$, there is an inclusion
  \[
    \iota_{I,J}\colon K_{I\sqcup J}\to K_I*K_J,\quad\sigma\mapsto(\sigma\cap I,\sigma\cap J).
  \]
  Baskakov, Buchstaber and Panov \cite{BBP} proved:

  \begin{lemma}
    \label{Hochster}
    There is an isomorphism of vector spaces
   \[
     \mathcal{H}_i(\F[K])\cong\bigoplus_{\emptyset\ne I\subset[m]}\widetilde{H}^{i-|I|-1}(K_I;\F)
   \]
   for $i>0$ such that for non-empty subsets $I,J\subset[m]$ the multiplication
   \[
      \widetilde{H}^{i-|I|-1}(K_I;\F)\otimes\widetilde{H}^{i-|J|-1}(K_J;\F)\to\widetilde{H}^{i+j-|I\cup J|-1}(K_{I\cup J};\F)
   \]
   is trivial for $I\cap J\ne\emptyset$ and the induced map of $\iota_{I,J}$ for $I\cap J=\emptyset$.
  \end{lemma}

  Let $M$ be a triangulation of a closed connected $\F$-oriented $d$-manifold with vertex set $[m]$. We consider a relation between the inclusion $\iota_{I,J}$ and the Poincar\'e duality. For any subset $I\subset[m]$, the Poincar\'e duality \cite[Proposition 3.46]{H} holds such that the map
  \[
    H^i(|M_I|;\F)\to H_{d-i}(|M|,|M|-|M_I|;\F),\quad\alpha\mapsto\alpha\frown[M]
  \]
  is an isomorphism, where $[M]$ denotes the fundamental class of $M$. By \cite[Lemma 70.1]{M}, $|M|-|M_I|\simeq|M_J|$ for $J=[m]-I$. Then there is an isomorphism
  \[
    D_{I,J}\colon H^i(M_I;\F)\xrightarrow{\cong}H_{d-i}(M,M_J;\F).
  \]
  Let $\partial\colon H_*(M,M_J;\F)\to H_{*-1}(M_J;\F)$ denote the boundary map of the long exact sequence
  \[
    \cdots\to H_*(M_J;\F)\to H_*(M;\F)\to H_*(M,M_J;\F)\xrightarrow{\partial}H_{*-1}(M_J;\F)\to\cdots.
  \]

  \begin{lemma}
    \label{Poincare}
    Let $M$ be a triangulation of a closed connected $\F$-oriented $d$-manifold with vertex set $[m]$. For any partition $[m]=I\sqcup J$ and $\alpha\in H^i(M_I;\F)$,
    \[
      (\partial\circ D_{I,J})(\alpha)=(-1)^{i+1}(\alpha\otimes 1)((\iota_{I,J})_*([M]))\in H_{d-i-1}(K_J;\F)
    \]
    where we regard $(\iota_{I,J})_*([M])$ as an element of $\bigoplus_{i+j=d-1}H_i(M_I;\F)\otimes H_j(M_J;\F)\cong H_d(M_I*M_J;\F)$.
  \end{lemma}

  \begin{proof}
    Let $\varphi\in C^i(M_I;\F)$ be a representative of $\alpha$. We define $\overline{\varphi}\in C^i(M;\F)$ by
    \[
      \overline{\varphi}(\sigma)=
      \begin{cases}
        \varphi(\sigma)&\sigma\in M_I\\
        0&\text{otherwise.}
      \end{cases}
    \]
    Then $\alpha\frown[M]$ is represented by $\overline{\varphi}\frown\mu$ where $\mu$ represents $[M]$. Let $[v_0,\ldots,v_i]$ denote an oriented $i$-simplex with vertices $v_0,\ldots,v_i$. We may set
    \[
      \mu=\sum_ka_k[v_0^k,v_1^k,\ldots,v_d^k]\in C_d(M;\F)
    \]
    for $a_k\in\F$, where $v_0^k,\ldots,v_{n_k}^k\in I$ and $v_{n_k+1}^k,\ldots,v_d^k\in J$ for some $n_k$. Then $(\partial\circ D_{I,J})(\alpha)$ is represented by
    \[
      \partial(\overline{\varphi}\frown\mu)=(\overline{\varphi}\circ\partial)\frown\mu=[(\overline{\varphi}\circ\partial)\frown\mu]=\sum_ka_k\overline{\varphi}(\partial[v_0^k,\ldots,v_{i+1}^k])[v_{i+1}^k,\ldots,v_d^k].
    \]
    Since $(\overline{\varphi}\circ\partial)\vert_{C_{i+1}(M_I;\F)}=\varphi\circ\partial=0$, $\overline{\varphi}(\partial[v_0^k,\ldots,v_{i+1}^k])\ne 0$ only when $n_k=i$. Then $(\partial\circ D_{I,J})(\alpha)$ is represented by
    \[
      \sum_{n_k=i}a_k\overline{\varphi}(\partial[v_0^k,\ldots,v_{i+1}^k])[v_{i+1}^k,\ldots,v_d^k]=(-1)^{i+1}\sum_{n_k=i}a_k\varphi([v_0^k,\ldots,v_i^k,\widehat{v_{i+1}^k}])[v_{i+1}^k,\ldots,v_d^k]
    \]
    On the other hand, since the $C_i(M_I;\F)\otimes C_{d-i-1}(M_J;\F)$ part of $\mu$ is $\sum_{n_k=i}a_k[v_0^k,\ldots,v_d^k]$, $(\iota_{I,J})_*([M])$ is represented by
    \[
      \sum_{n_k=i}a_k[v_0^k,\ldots,v_i^k]\otimes[v_{i+1}^k,\ldots,v_d^k].
    \]
    Thus the proof is complete.
  \end{proof}

  Now we are ready to prove:

  \begin{theorem}
    \label{preGolod}
    If a triangulation of a closed connected $\F$-orientable $d$-manifold is weakly $\F$-Golod, then it is $\F$-tight.
  \end{theorem}

  \begin{proof}
    Let $M$ be a triangulation of a closed connected $\F$-oriented $d$-manifold with vertex set $[m]$. Let $[m]=I\sqcup J$ be a partition. Suppose that the map $\iota_{I,J}$ is trivial in cohomology with coefficients in $\F$. Then by the universal coefficient theorem, $\iota_{I,J}$ is trivial in homology with coefficients in $\F$ too. Thus by Lemma \ref{Poincare} the boundary map $\partial\colon H_*(M,M_J;\F)\to H_{*-1}(M_J;\F)$ is trivial, and so the natural map $H_*(M_J;\F)\to H_*(M;\F)$ is injective, completing the proof.
  \end{proof}


  \section{The complex $F(M)$}\label{F(M)}

  Throughout this section, let $M$ be a closed connected tight-neighborly $d$-manifold triangulation for $d\ge 3$ with vertex set $[m]$. Let $K$ be a simplicial complex with vertex set $[m]$. A subset $I\subset[m]$ is a \emph{minimal non-face} of $K$ if every proper subset of $I$ is a simplex of $K$ and $I$ itself is not a simplex of $K$. Define a simplicial complex $F(M)$ by filling all minimal non-faces of cardinality $d+1$ into $M$. This section investigates the complex $F(M)$.

  We set notation. Let $K$ be a simplicial complex with vertex set $[m]$. The link of a vertex $v$ in a simplicial complex $K$ is defined by
  \[
    \lk_K(v)=\{\sigma\in K\mid v\not\in\sigma\text{ and }\sigma\sqcup v\in K\}.
  \]
  For a finite set $S$, let $\Delta(S)$ denote the simplex with vertex set $S$. Then $I\subset[m]$ is a minimal non-face of $K$ if and only if $K_I=\partial\Delta(I)$. Let $K_1,K_2$ be simplicial complexes of dimension $d$ such that $K_1\cap K_2$ is a single $d$-simplex $\sigma$. Then we write
  \[
    K_1\# K_2=K_1\cup K_2-\sigma\quad\text{and}\quad K_1\circ K_2=K_1\cup K_2.
  \]

  \begin{lemma}
    \label{link}
    For each $v\in [m]$, there are $V(v,1),\ldots,V(v,n_v)\subset[m]$ such that $|V(v,k)|=d+1$ for $1\le k\le n_v$ and
    \[
      \lk_M(v)=\partial\Delta(V(v,1))\#\cdots\#\partial\Delta(V(v,n_v)).
    \]
  \end{lemma}

  \begin{proof}
    The case $d=3$ is proved \cite[Proof of Theorem 1.2]{BDS}. For $d\ge 4$, tight-neighborliness implies local stackedness, that is, every vertex link is a stacked sphere, as in \cite{DM}. Moreover, stacked spheres are characterized by Bagchi and Datta \cite{BD} such that every stacked $(d-1)$-sphere is of the form $\partial\Delta^d\#\cdots\#\partial\Delta^d$. Then we obtain the result for $d\ge 4$.
  \end{proof}

  Generalizing neighborliness, we say that a simplicial complex is \emph{$k$-neighborly} if every $k+1$ vertices form a simplex. So 1-neighborliness is precisely neighborliness.

  \begin{lemma}
    \label{link neighborly}
    For each $v\in [m]$ and $1\le k\le n_v$, $M_{V(v,k)\sqcup v}$ is $(d-1)$-neighborly.
  \end{lemma}

  \begin{proof}
    By Lemma \ref{link}, $\lk_M(v)_{V(v,k)}$ is $\partial\Delta^d$ removed some $(d-1)$-simplices, implying it is $(d-2)$-neighborly. So if $I$ is a subset of $V(v,k)$ with $|I|=d-1$, then $I\sqcup v$ is a simplex of $M$. It remains to show $M_{V(v,k)}$ is $(d-1)$-neighborly. Let $J$ be any subset of $V(v,k)$ with $|J|=d$. Then $\partial\Delta(J)$ is a subcomplex of $M$. If $M_J=\partial\Delta(J)$, then $M_{J\sqcup v}=\partial\Delta(J)*v$, which is contractible. So the inclusion $M_J\to M_{J\sqcup v}$ is not injective in homology with coefficients in $\F$. By Theorem \ref{tight vs tight-neighborly}, $M$ is $\F$-tight, so we get a contradiction. Thus $J$ must be a simplex of $M$, completing the proof.
  \end{proof}

  We prove local properties of the complex $F(M)$.

  \begin{proposition}
    \label{local F}
    \begin{enumerate}
      \item For each $v\in[m]$,
      \[
        \lk_{F(M)}(v)=\partial\Delta(V(v,1))\circ\cdots\circ\partial\Delta(V(v,n_v)).
      \]

      \item For each $v\in[m]$ and $1\le k\le n_v$, $V(v,k)\sqcup v$ is a minimal non-face of $F(M)$.
    \end{enumerate}
  \end{proposition}

  \begin{proof}
    (1) Let $\sigma$ be the $(d-1)$-simplex $(\partial\Delta(V(v,1))\#\cdots\#\partial\Delta(V(v,k)))\cap\partial\Delta(V(v,k+1))$. Then by Lemma \ref{link neighborly}, $\partial\Delta(\sigma\sqcup v)$ is a subcomplex of $M$, implying $\sigma\sqcup v$ is a simplex of $F(M)$. Then by induction, we get $\partial\Delta(V(v,1))\circ\cdots\circ\partial\Delta(V(v,n_v))\subset\lk_{F(M)}(v)$. The reverse inclusion is obvious by the construction of $F(M)$, completing the proof.

    (2) By Lemma \ref{link neighborly}, $V(v,k)$ is a simplex of $F(M)$, so every proper subset $I$ of $V(v,k)\sqcup v$ is a simplex of $F(M)$. By (1), $V(v,k)\sqcup v$ is not a simplex of $F(M)$. Then the statement is proved.
  \end{proof}

  We compute the homology of $F(M)$. Let
  \[
    S(M)=\{V(v,k)\sqcup v\mid v\in[m]\text{ and }1\le k\le n_v\}.
  \]
  Then $S(M)$ is the set of all subsets $I\subset[m]$ such that $|I|=d+2$ and $\lk_{M_I}(v)$ is $(d-2)$-neighborly for some $v\in I$.

  \begin{lemma}
    \label{S}
    $F(M)=\bigcup_{I\in S(M)}\partial\Delta(I)$.
  \end{lemma}

  \begin{proof}
    Let $K=\bigcup_{I\in S(M)}\partial\Delta(I)$. By Proposition \ref{local F}, $K\subset F(M)$. For any $k$-simplex $\sigma$ of $F(M)$ with $0\le k\le d-1$ and $v\in\sigma$, $\sigma-v$ is a simplex of $\lk_M(v)$ because $\sigma$ is a simplex of $M$ too. Then $\sigma-v\subset V(v,l)$ for some $1\le l\le n_v$, implying $\sigma$ is a simplex of $K$. Thus the $(d-1)$-skeleton of $F(M)$ is included in $K$. Take any $d$-simplex $\sigma$ of $F(M)$. Then $\sigma$ is either a simplex or a minimal non-face of $M$. In both cases, $\partial\Delta(\sigma-v)$ is a subcomplex of $\lk_M(v)$ for $v\in\sigma$. Then $\sigma-v\subset V(v,l)$ for some $1\le l\le n_v$, implying $\sigma$ is a simplex of $K$. Thus $F(M)\subset K$, completing the proof.
  \end{proof}

  We compute the homology of $F(M)$. By Lemma \ref{S}, there is an inclusion $g_I\colon\partial\Delta(I)\to F(M)$ for each $I\in S(M)$. Let $u_I\in H_d(F(M);\Z)$ be the Hurewicz image of $g_I$.

  \begin{proposition}
    \label{homology}
    The integral homology of $F(M)$, except for dimension 1, is given by
    \[
      \widetilde{H}_*(F(M);\Z)=
      \begin{cases}
        \Z\langle u_I\mid I\in S(M)\rangle&*=d\\
        0&*\ne 1,d.
      \end{cases}
    \]
  \end{proposition}

  \begin{proof}
    Since $F(M)$ is obtained from $M$ by attaching $d$-simplices, we only need to calculate $H_{d-1}$ and $H_d$ by Theorem \ref{tight vs tight-neighborly}. By Lemma \ref{link}, each component of $\lk_{M_I}(v)$ is $(d-2)$-connected, where $\lk_{M_I}(v)=\lk_M(v)_{I-v}$. Then there is an exact sequence
    \begin{multline*}
      0\to\widetilde{H}_d(F(M)_{I-v};\Z)\to H_d(F(M)_I;\Z)\xrightarrow{\partial}H_{d-1}(\lk_{F(M)_I}(v);\Z)\\
      \to H_{d-1}(F(M)_{I-v};\Z)\to H_{d-1}(F(M)_I;\Z)\to 0.
    \end{multline*}
    By Proposition \ref{local F}, there is an inclusion $\partial\Delta(V(v,k))\to\lk_{F(M)_I}(v)$ for $V(v,k)\sqcup v\subset I$, and we write the Hurewicz image of this inclusion by $\bar{u}_{V(v,k)}$. By Proposition \ref{local F},
    \[
      H_{d-1}(\lk_{F(M)_I}(v);\Z)=\Z\langle \bar{u}_{V(v,k)}\mid V(v,k)\sqcup v\subset I\rangle
    \]
    such that $\partial(u_{V(v,k)\sqcup v})=\bar{u}_{V(v,k)}$. Then $\partial$ is surjective, so we get an isomorphism
    \[
      H_{d-1}(F(M)_{I-v};\Z)\cong H_{d-1}(F(M)_I;\Z).
    \]
    Thus we obtain $H_{d-1}(F(M)_I;\Z)=0$ for any $I\subset[m]$ by induction on $|I|$, where $H_{d-1}(F(M)_I;\Z)=0$ for $|I|=1$. We also get a split exact sequence
    \[
      0\to H_d(F(M)_{I-v};\Z)\to H_d(F(M)_I;\Z)\xrightarrow{\partial}H_{d-1}(\lk_{F(M)_I}(v);\Z)\to 0.
    \]
    Then by induction on $|I|$, we also obtain
    \[
      H_d(F(M)_I;\Z)=\Z\langle u_{V(v,k)}\mid V(v,k)\sqcup v\subset I\rangle.
    \]
    Thus the proof is complete.
  \end{proof}

  By Theorem \ref{tight vs tight-neighborly}, $\pi_1(|M|)$ is a free group. Since $|F(M)|$ is obtained by attaching $d$-cells to $|M|$, the inclusion $|M|\to|F(M)|$ is an isomorphism in $\pi_1$, so $\pi_1(|F(M)|)$ is a free group too. Then there is a map $f\colon B\to|F(M)|$ which is an isomorphism in $\pi_1$, where $B$ is a wedge of circles. Let $\widehat{F}(M)$ be the cofiber of $f$. Since there is an exact sequence
  \[
    \cdots\to H_*(B;\Z)\xrightarrow{f_*}H_*(F(M);\Z)\to \widetilde{H}_*(\widehat{F}(M);\Z)\to\cdots
  \]
  the natural map $H_*(F(M);\Z)\to H_*(\widehat{F}(M);\Z)$ is an isomorphism for $*\ne 1$. Let $\hat{g}_I$ be the composite $|\partial\Delta(I)|\xrightarrow{g_I}|F(M)|\to\widehat{F}(M)$ for $I\in S(M)$, and let $\hat{u}_I$ be the Hurewicz image of $\hat{g}_I$. By Proposition \ref{homology}, we get:

  \begin{corollary}
    \label{homology hat}
    The integral homology of $\widehat{F}(M)$ is given by
    \[
      \widetilde{H}_*(\widehat{F}(M);\Z)=
      \begin{cases}
        \Z\langle\hat{u}_I\mid I\in S(M)\rangle&*=d\\
        0&*\ne d.
      \end{cases}
    \]
  \end{corollary}

  Since $\widehat{F}(M)$ is path-connected, there is a map
  \[
    g\colon \bigvee_{I\in S(M)}|\partial\Delta(I)|\to\widehat{F}(M)
  \]
  such that $g\vert_{|\partial\Delta(I)|}\simeq\hat{g}_I$ for each $I\in S(M)$. Then by Corollary \ref{homology hat} and the J.H.C. Whitehead theorem, we obtain the following.

  \begin{corollary}
    \label{g hat}
    The map $g\colon \bigvee_{I\in S(M)}|\partial\Delta(I)|\to\widehat{F}(M)$ is a homotopy equivalence.
  \end{corollary}


  \section{Polyhedral product}\label{polyhedral product}

  Throughout this section, let $K$ be a simplicial complex with vertex set $[m]$. Let $(\underline{X},\underline{A})=\{(X_i,A_i)\}_{i=1}^m$ be a collection of pairs of pointed spaces indexed by vertices of $K$. For $I\subset[m]$, let
  \[
    (\underline{X},\underline{A})^I=Y_1\times\cdots\times Y_m
  \]
  where $Y_i=X_i$ for $i\in I$ and $Y_i=A_i$ for $i\not\in I$. The \emph{polyhedral product} of $(\underline{X},\underline{A})$ over $K$ is defined by
  \[
    Z_K(\underline{X},\underline{A})=\bigcup_{\sigma\in K}(\underline{X},\underline{A})^\sigma.
  \]
  For $\emptyset\ne I\subset[m]$, let $(\underline{X}_I,\underline{A}_I)=\{(X_i,A_i)\}_{i\in I}$. Then we can define $Z_{K_I}(\underline{X}_I,\underline{A}_I)$. The following lemma is immediate from the definition of a polyhedral product.

  \begin{lemma}
    \label{retract}
    For each $\emptyset\ne I\subset[m]$, $Z_{K_I}(\underline{X}_I,\underline{A}_I)$ is a retract of $Z_K(\underline{X},\underline{A})$.
  \end{lemma}

  For a collection of pointed spaces $\underline{X}=\{X_i\}_{i=1}^m$, let $(C\underline{X},\underline{X})=\{(CX_i,X_i)\}_{i=1}^m$. For $0\le i\le m$, we define a subspace of $Z_K(C\underline{X},\underline{X})$ by
  \[
    Z_K^i(C\underline{X},\underline{X})=\{(x_1,\ldots,x_m)\in Z_K(C\underline{X},\underline{X})\mid\text{at least }m-i\text{ of }x_1,\ldots,x_m\text{ are basepoints}\}.
  \]
  Using the basepoint of each $X_i$, we regard $Z_{K_I}(C\underline{X}_I,\underline{X}_I)$ as a subspace of $Z_K(C\underline{X},\underline{X})$ so that we can alternatively write
  \begin{equation}
    \label{union}
    Z_K^i(C\underline{X},\underline{X})=\bigcup_{I\subset[m],\,|I|=i}Z_{K_I}(C\underline{X}_I,\underline{X}_I).
  \end{equation}
  There is a filtration
  \[
    *=Z_K^0(C\underline{X},\underline{X})\subset Z_K^1(C\underline{X},\underline{X})\subset\cdots\subset Z_K^m(C\underline{X},\underline{X})=Z_K(C\underline{X},\underline{X})
  \]
  which we call the \emph{fat-wedge filtration} of $Z_K(C\underline{X},\underline{X})$. By \cite[Theorem 4.1]{IK3}, we have
  \[
    Z_K^i(C\underline{X},\underline{X})/Z_K^{i-1}(C\underline{X},\underline{X})=\bigvee_{I\subset[m],\,|I|=i}|\Sigma K_I|\wedge\widehat{X}^I
  \]
  where $\widehat{X}^I=\bigwedge_{i\in I}X_i$. Moreover, it is shown in \cite[Corollary 4.2]{IK3} that the fat-wedge filtration of $Z_K(C\underline{X},\underline{X})$ splits after a suspension and the decomposition of Bahri, Bendersky, Cohen and Gitler \cite[Theorem 2.2.1]{BBCG} is reproduced as:

  \begin{theorem}
    [BBCG decomposition]
    \label{BBCG}
    There is a homotopy equivalence
    \[
      \Sigma Z_K(C\underline{X},\underline{X})\simeq\Sigma\bigvee_{\emptyset\ne I\subset[m]}|\Sigma K_I|\wedge\widehat{X}^I.
    \]
  \end{theorem}

  In particular, if the BBCG decomposition desuspends, then $Z_K(C\underline{X},\underline{X})$ itself desuspends. Moreover, if each $X_i$ is a connected CW complex, then the BBCG decomposition desuspends whenever $Z_K(C\underline{X},\underline{X})$ desuspends \cite{IK3}. Then we aim to desuspend the BBCG decomposition. Desuspension of the BBCG decomposition was studied for specific Golod complexes such as shifted complexes \cite{IK0,GT,GW}, and desuspension for much broader classes of simplicial complexes, including the previous specific simplicial complexes, can be reproved by using the fat-wedge filtration technique \cite{IK3}.

  The moment-angle complex $Z_K$ introduced in Section \ref{introduction} is the polyhedral product $Z_K(D^2,S^1)$. The \emph{real moment-angle complex} $\R Z_K$ is defined to be the polyhedral product $Z_K(D^1,S^0)$, and we denote its fat-wedge filtration by
  \[
    *=\R Z^0_K\subset\R Z^1_K\subset\cdots\subset\R Z^m_K=\R Z_K
  \]
  where we choose the basepoint of $S^0=\{-1,+1\}$ to be $-1$. The fat-wedge filtration of $\R Z_K$ is proved to be a cone decomposition \cite[Theorem 3.1]{IK3}. For $\emptyset\ne I\subset[m]$, let $j_{K_I}\colon\R Z_{K_I}^{|I|-1}\to\R Z_K^{|I|-1}$ denote the inclusion.

  \begin{theorem}
    \label{FWF RZ}
    For each $\emptyset\ne I\subset[m]$, there is a map $\varphi_{K_I}\colon|K_I|\to\R Z^{|I|-1}_{K_I}$ such that
    \[
      \R Z_K^i=\R Z_K^{i-1}\bigcup_{I\subset[m],\,|I|=i}C|K_I|
    \]
    where the attaching maps are $j_{K_I}\circ\varphi_{K_I}$.
  \end{theorem}

  We say that the fat-wedge filtration of $\R Z_K$ is trivial if $\varphi_{K_I}$ is null homotopic for each $\emptyset\ne I\subset[m]$. We remark that $\varphi_{K_I}$ is null homotopic if and only if so is $j_{K_I}\circ\varphi_{K_I}$ because $\R Z_{K_I}^{|I|-1}$ is a retract of $\R Z_K^{|I|-1}$. The fat-wedge filtration is useful for desuspending the BBCG decomposition because we have the following criterion \cite[Theorem 1.2]{IK3}.

  \begin{theorem}
    \label{FWF decomposition}
    If the fat-wedge filtration of $\R Z_K$ is trivial then for any $\underline{X}$,
    \[
      Z_K(C\underline{X},\underline{X})\simeq\bigvee_{\emptyset\ne I\subset[m]}|\Sigma K_I|\wedge\widehat{X}^I.
    \]
  \end{theorem}

  For $\emptyset\ne I\subset[m]$, define a map $\alpha_I\colon\R Z_{K_I}^{|I|-1}\to\R Z_K^{m-1}$ by $\alpha_I(x_i\mid i\in I)=(y_1,\ldots,y_m)$ such that
  \[
    y_i=
    \begin{cases}
      x_i&i\in I\\
      +1&i\not\in I
    \end{cases}
  \]
  for $(x_i\mid i\in I)\in\R Z_{K_I}^{|I|-1}$. Note that $\alpha_I$ is not the natural inclusion because the basepoint of $S^0=\{-1,+1\}$ is taken to be $-1$ as mentioned above. For $\emptyset\ne J\subset I\subset[m]$ and $|J|\le i\le |I|$, let $\pi$ denote the composite of projections
  \[
    \R Z_{K_I}^i\to\R Z_{K_J}\to\R Z_{K_J}/\R Z_{K_J}^{|J|-1}=|\Sigma K_J|.
  \]
  By the construction of $\varphi_K$, we have:

  \begin{lemma}
    \label{FWF naturality}
    For $\emptyset\ne J\subsetneq I\subset[m]$, there is a commutative diagram
    \[
      \xymatrix{
        |K_I|\ar[r]^{\varphi_{K_I}}\ar[d]&\R Z_{K_I}^{|I|-1}\ar[r]^\pi\ar[d]^{\alpha_I}&|\Sigma K_J|\ar[d]^{|\Sigma j|}\\
        |K|\ar[r]^{\varphi_K}&\R Z_K^{m-1}\ar[r]^(.4)\pi&|\Sigma K_{J\sqcup([m]-I)}|
      }
    \]
    where $j\colon K_J\to K_{J\sqcup([m]-I)}$ is the inclusion.
  \end{lemma}

  The following two lemmas, proved in \cite[Proof of Theorem 7.2]{IK3} and \cite[Lemma 10.1]{IK3} respectively, are quite useful in detecting the triviality of $\varphi_K$.

  \begin{lemma}
    \label{FWF minimal non-face}
    Let $\overline{K}$ be a simplicial complex obtained by filling all minimal non-faces into $K$. Then $\varphi_K$ factors through the inclusion $|K|\to|\overline{K}|$.
  \end{lemma}

  \begin{lemma}
    \label{FWF pinch}
    If $\varphi_{K_I}\simeq *$ for each $\emptyset\ne I\subsetneq[m]$, then the composite
    \[
      |K|\xrightarrow{\varphi_K}\R Z_K^{m-1}\to \R Z_{K_J}\xrightarrow{\pi}|\Sigma K_J|
    \]
    is null-homotopic for each $\emptyset\ne J\subsetneq[m]$.
  \end{lemma}

  Finally, we estimate the connectivity of $\R Z_K$.

  \begin{lemma}
    \label{connectivity}
    If $K$ is $k$-neighborly, then $\R Z_K$ is $k$-connected.
  \end{lemma}

  \begin{proof}
    The proof is done by the same calculation as \cite[Proposition 5.3]{IK3}. The proof is alternatively done as follows. By definition, $\pi_*(\R Z_K)$ is isomorphic to $\pi_*(\R Z_{K_k})$ for $*\le k$, where $K_k$ denotes the $k$-skeleton of $K$. Since $K$ is $k$-neighborly, $K_k=\Delta^{m-1}_k$. Since $\Delta^{m-1}_k$ is shifted, it follows from \cite{IK0} that there is a homotopy equivalence
    \[
      \R Z_{\Delta_k^{m-1}}\simeq\bigvee_{\emptyset\ne I\subset[m]}|\Sigma(\Delta^{m-1}_k)_I|.
    \]
    Since each $|\Sigma(\Delta^{m-1}_k)_I|$ is $k$-connected, the proof is done.
  \end{proof}







  \section{Proof of Theorem \ref{main general}}\label{proof}

  Throughout this section, let $M$ be a tight-neighborly triangulation of a closed connected $\F$-orientable $d$-manifold with vertex set $[m]$ unless otherwise is specified. We aim to prove that the fat-wedge filtration of $\R Z_M$ is trivial. First, we compute the fundamental group of $|F(M)_I|$ for $\emptyset\ne I\subset[m]$.

  \begin{lemma}
    \label{pi_1}
    For each $\emptyset\ne I\subset[m]$, $\pi_1(|F(M)_I|)$ is a free group.
  \end{lemma}

  \begin{proof}
    Since the fundamental group of a suspension is a free group, we prove $|F(M)_I|$ is a suspension by induction on $I$. For $|I|=1$, $|F(M)_I|$ is obviously a suspension. Suppose that $|F(M)_{I-v}|$ is a suspension for $v\in I$. Note that
    \begin{equation}
      \label{dl-lk}
      F(M)_I=F(M)_{I-v}\cup(\lk_{F(M)_I}(v)*v)
    \end{equation}
    where $F(M)_{I-v}\cap(\lk_{F(M)_I}(v)*v)=\lk_{F(M)_I}(v)$. Since $\lk_{F(M)_I}(v)=\lk_{F(M)}(v)_{I-v}$, it follows from Proposition \ref{local F} that there are inclusions
    \[
      \lk_{F(M)_I}(v)\subset(\Delta(V(v,1))\circ\cdots\circ\Delta(V(v,n_v)))_{I-v}\subset F(M)_{I-v}.
    \]
    Since $M$ is neighborly by Theorem \ref{tight vs tight-neighborly}, so is $M_{I-v}$, implying $F(M)_{I-v}$ is connected. On the other hand, each component of $(\Delta(V(v,1))\circ\cdots\circ\Delta(V(v,n_v)))_{I-v}$ is contractible. Then the inclusion $|(\Delta(V(v,1))\circ\cdots\circ\Delta(V(v,n_v)))_{I-v}|\to|F(M)_{I-v}|$ is null-homotopic, and so the inclusion $|\lk_{F(M)_I}(v)|\to|F(M)_{I-v}|$ is null-homotopic too. Thus by \eqref{dl-lk}, we get a homotopy equivalence
    \[
      |F(M)_I|\simeq|F(M)_{I-v}|\vee|\Sigma\lk_{F(M)_I}(v)|.
    \]
    Since $|F(M)_{I-v}|$ is a suspension by the induction hypothesis, $|F(M)_I|$ turns out to be a suspension, completing the proof.
  \end{proof}

  Let $\emptyset\ne I\subset[m]$. By Lemma \ref{FWF minimal non-face}, the map $\varphi_{M_I}$ decomposes as
  \begin{equation}
    \label{varphi}
    |M_I|\to |F(M)_I|\to\R Z_{M_I}^{|I|-1}.
  \end{equation}
  By Lemma \ref{pi_1}, there is a map $f_I\colon B_I\to|F(M)_I|$, where $B_I$ is a wedge of circles, such that $f_I$ is an isomorphism in $\pi_1$. Let $\widehat{F}(M)_I$ denote the cofiber of $f_I$, where $\widehat{F}(M)_{[m]}$ coincides with $\widehat{F}(M)$ in Section \ref{F(M)}. On the other hand, since $M$ is neighborly by Lemma \ref{tight neighborly}, so is $M_J$ for any $\emptyset\ne J\subset[m]$. Then by \eqref{union} and Lemma \ref{connectivity}, we can see that $\R Z_{M_I}^{|I|-1}$ is simply-connected. In particular, there is a commutative diagram
  \begin{equation}
    \label{naturality 1}
    \xymatrix{
      |F(M)_I|\ar[r]\ar[d]&\widehat{F}(M)_I\ar[d]\\
      \R Z_{M_I}^{|I|-1}\ar@{=}[r]&\R Z_{M_I}^{|I|-1}.
    }
  \end{equation}
  Then by combining \eqref{varphi} and \eqref{naturality 1}, we get:

  \begin{lemma}
    \label{varphi decomp}
    For each $\emptyset\ne I\subset[m]$, the map $\varphi_{M_I}$ factors through the inclusion $|M_I|\to\widehat{F}(M)_I$.
  \end{lemma}

  \begin{proposition}
    \label{I<m}
    For each $\emptyset\ne I\subsetneq[m]$, the map $\varphi_{M_I}$ is null-homotopic.
  \end{proposition}

  \begin{proof}
    As is computed in the proof of Proposition \ref{homology}, $\widetilde{H}_*(F(M)_I;\Z)=0$ unless $*=1,d$. Thus as well as $\widehat{F}(M)$, we can see that $\widehat{F}(M)_I$ is $(d-1)$-connected. Since $I\ne[m]$, $|M_I|$ is homotopy equivalent to a CW complex of dimension $\le d-1$. Then we obtain that the inclusion $|M_I|\to\widehat{F}(M)_I$ is null-homotopic. Thus by Lemma \ref{varphi decomp}, the proof is complete.
  \end{proof}

  It remains to show $\varphi_M$ is null-homotopic. By Lemma \ref{FWF naturality}, there is a commutative diagram
  \[
    \xymatrix{
      \bigvee_{I\in S(M)}|M_I|\ar[rr]\ar[d]_{\bigvee_{I\in S(M)}\varphi_{M_I}}&&|M|\ar[d]^{\varphi_M}\\
      \bigvee_{I\in S(M)}\R Z_{M_I}^{d+1}\ar[rr]^(.58){\bigvee_{I\in S(M)}\alpha_I}&&\R Z_M^{m-1}.
    }
  \]
  Then since $F(M)_I=\partial\Delta(I)$ for $I\in S(M)$ by Proposition \ref{local F}, we get a commutative diagram
  \begin{equation}
    \label{naturality 2}
    \xymatrix{
      \bigvee_{I\in S(M)}|\partial\Delta(I)|\ar[rr]^(.6){\bigvee_{I\in S(M)}g_I}\ar[d]&&|F(M)|\ar[d]\\
      \bigvee_{I\in S(M)}\R Z_{M_I}^{d+1}\ar[rr]^(.6){\bigvee_{I\in S(M)}\alpha_I}&&\R Z_M^{m-1}.
    }
  \end{equation}
  Juxtaposing the commutative diagrams \eqref{naturality 1} and \eqref{naturality 2}, we get a commutative diagram
  \[
    \xymatrix{
      \bigvee_{I\in S(M)}|\partial\Delta(I)|\ar[rr]^(.6)g\ar[d]&&\widehat{F}(M)\ar[d]\\
      \bigvee_{I\in S(M)}\R Z_{M_I}^{d+1}\ar[rr]^(.6){\bigvee_{I\in S(M)}\alpha_I}&&\R Z_M^{m-1}
    }
  \]
  and by Corollary \ref{g hat} and Lemma \ref{varphi decomp}, we obtain:

  \begin{lemma}
    \label{phi decomposition}
    The map $\varphi_M\colon|M|\to\R Z_M^{m-1}$ is homotopic to the composite
    \[
      |M|\to\widehat{F}(M)\xrightarrow{g^{-1}}\bigvee_{I\in S(M)}|\partial\Delta(I)|\to\bigvee_{I\in S(M)}\R Z_{M_I}^{d+1}\xrightarrow{\bigvee_{I\in S(M)}\alpha_I}\R Z_M^{m-1}.
    \]
  \end{lemma}

  We will investigate the composite of maps in Lemma \ref{phi decomposition} by identifying a homotopy set with a homology.

  \begin{lemma}
    \label{Hopf}
    Let $W$ be a finite wedge of $S^d$. Then there is an isomorphism of sets
    \[
      [|M|,W]\cong H^d(M;\Z)\otimes H_d(W;\Z)
    \]
    which is natural with respect to maps among finite wedges of $S^d$.
  \end{lemma}

  \begin{proof}
    Since $\dim M=d$, the statement follows from the Hopf degree theorem.
  \end{proof}

  \begin{lemma}
    \label{injective}
    For each $v\in I\in S(M)$, the natural map
    \[
      H^d(M;\Z)\otimes H_{d-1}(M_{I-v};\Z)\to H^d(M;\Z)\otimes H_{d-1}(M_{[m]-v};\Z)
    \]
    is injective.
  \end{lemma}

  \begin{proof}
    By Lemma \ref{link neighborly}, $|M_{I-v}|$ is contractible or $S^{d-1}$, so in particular, $H_{d-1}(M_{I-v};\Z)$ is a free abelian group, and so there is a natural isomorphism
    \begin{equation}
      \label{H_2 K_I}
      H_{d-1}(M_{I-v};\F)\cong H_{d-1}(M_{I-v};\Z)\otimes\F.
    \end{equation}

    By definition, $M_{[m]-v}$ is $M$ removed the open star of $v$, which is homotopy equivalent to $M$ removed the vertex $v$ by \cite[Lemma 70.1]{M}. Then by Theorem \ref{tight vs tight-neighborly}, $|M_{[m]-v}|$ is homotopy equivalent to a wedge of finitely many, possibly zero, copies of $S^1$ and $S^{d-1}$. Then $H_*(M_{[m]-v};\Z)$ is a free abelian group, and so there is a natural isomorphism
    \begin{equation}
      \label{H_2 K}
      H_{d-1}(M_{[m]-v};\F)\cong H_{d-1}(M_{[m]-v};\Z)\otimes\F.
    \end{equation}

    Since $M$ is $\F$-tight by Theorem \ref{tight vs tight-neighborly}, the natural map $H_{d-1}(M_{I-v};\F)\to H_{d-1}(M_{[m]-v};\F)$ is injective. Then by \eqref{H_2 K_I} and \eqref{H_2 K}, the natural map
    \[
      H_{d-1}(M_{I-v};\Z)\otimes\F\to H_{d-1}(M_{[m]-v};\Z)\otimes\F
    \]
    is injective too. Since both $H_{d-1}(M_{I-v};\Z)$ and $H_{d-1}(M_{[m]-v};\Z)$ are free abelian group, the case that $M$ is orientable is proved because $H^d(M;\Z)\cong\Z$. If $M$ is unorientable, then $H^d(M;\Z)\cong\F_2$ and the base field $\F$ is of characteristic 2, where $\F_2$ is the field of two elements. Thus the case that $M$ is unorientable is proved too, completing the proof.
  \end{proof}




  \begin{proposition}
    \label{trivial K}
    The map $\varphi_M\colon|M|\to\R Z_M^{m-1}$ is null-homotopic.
  \end{proposition}

  \begin{proof}

    Note that $m\ge d+2$. Let $\emptyset\ne J\subset I\in S(M)$. By Lemma \ref{link neighborly}, $|M_J|$ is contractible for $|J|\le d$, and $|M_J|$ is contractible or $S^{d-1}$ for $|J|=d+1$. Then by Proposition \ref{I<m}, there is a homotopy equivalence
    \begin{equation}
      \label{RZ}
      \R Z_{M_I}^{d+1}\simeq\bigvee_{v\in I}|\Sigma M_{I-v}|
    \end{equation}
    where $|\Sigma M_{I-v}|$ is contractible or $S^d$ as mentioned above. Let
    \[
      A=\bigvee_{I\in S(M)}\bigvee_{v\in I}|\Sigma M_{I-v}|\quad\text{and}\quad B=\bigvee_{I\in S(M)}\bigvee_{v\in I}|\Sigma M_{[m]-v}|
    \]
    where $A\simeq\bigvee_{I\in S(M)}\R Z_{M_I}^{d+1}$ by \eqref{RZ}. Let $f\colon|M|\to A$ denote the composite of the first three maps in Lemma \ref{phi decomposition}. Then it suffices to show $f$ is null-homotopic. By Lemma \ref{Hopf}, $f$ is identified with some element $\phi$ of $H_d(M;\Z)\otimes H_d(A;\Z)$, so that $f$ is null-homotopic if and only if $\phi=0$.

    As in the proof of Lemma \ref{injective}, $|\Sigma M_{[m]-v}|$ is a wedge of finitely many copies of $S^2$ and $S^d$ for each vertex $v$ of $M$. Let $C_v$ denote the $S^d$-wedge part of $|\Sigma M_{[m]-v}|$. Then there is a projection $q_v\colon B\to C_v$. By Lemmas \ref{FWF naturality}, \ref{FWF pinch} and \ref{phi decomposition}, the composite
    \begin{equation}
      \label{composite}
      |M|\xrightarrow{f}A\to|\Sigma M_{I-v}|\to |\Sigma M_{[m]-v}|
    \end{equation}
    is null homotopic for each $v\in I\in S(M)$. Then by Lemma \ref{Hopf}, $\phi$ is mapped to 0 by
    \[
      1\otimes(q_v\circ j)_*\colon H^d(M;\Z)\otimes H_d(A;\Z)\to H^d(M;\Z)\otimes H_d(C_v;\Z)
    \]
    for each $v\in I\in S(M)$, where $j\colon A\to B$ denotes the inclusion. Since the map
    \[
      \bigoplus_{v\in I\in S(M)}(q_v)_*\colon H_d(B;\Z)\to\bigoplus_{v\in I\in S(M)}H_d(C_v;\Z)
    \]
    is an isomorphism, we get $(1\otimes j_*)(\phi)=0$. Thus we obtain $\phi=0$ by Lemma \ref{injective}, completing the proof.
  \end{proof}

  Now we are ready to prove Theorem \ref{main general}:

  \begin{proof}
    [Proof of Theorem \ref{main general}]
    The implications (1) $\Rightarrow$ (2) $\Leftarrow$ (3) are proved by Theorems \ref{preGolod} and \ref{tight vs tight-neighborly}. The implication (3) $\Rightarrow$ (4) is proved by Propositions \ref{I<m} and \ref{trivial K}. If (4) holds, then by Theorem \ref{FWF decomposition}, $Z_M$ is a suspension. So by the fact that $K$ is $\F$-Golod whenever $Z_K$ is a suspension, as mentioned in Section \ref{introduction}, we obtain the implication (4) $\Rightarrow$ (1), completing the proof.
  \end{proof}


  \section{Further problems}\label{problems}

  This section proposes possible problems for a further study on a relationship of Golodness and tightness. As proved in Theorem \ref{preGolod}, weak $\F$-Golodness implies $\F$-tightness for a closed connected $\F$-orientable manifold triangulations. But not all weakly $\F$-Golod complexes are $\F$-tight. For example, if $K$ is the join of a vertex and the boundary of a simplex, then it is $\F$-Golod for any field $\F$ as the fat-wedge filtration of $\R Z_K$ is trivial but it is not $\F$-tight as in the proof of Lemma \ref{link neighborly}. Then it is natural to ask:

  \begin{problem}
    \label{problem 1}
    For which simplicial complexes does weak $\F$-Golodness imply $\F$-tightness?
  \end{problem}

  Interestingly, the opposite implication always holds.

  \begin{proposition}
      Let $K$ be a simplicial complex with vertex set $[m]$. If $K$ is $\F$-tight, then it is weakly $\F$-Golod.
  \end{proposition}

  \begin{proof}
    Take any disjoint subsets $\emptyset\ne I,J\subset[m]$. Then there is a map
    \[
      \iota_{I,J}\colon K_{I\sqcup J}\to K_I*K_J
    \]
    as in Section \ref{weak Golodness}. By Lemma \ref{Hochster}, $K$ is weakly $\F$-Golod if and only if the map $\iota_{I,J}$ is trivial in homology with coefficients in $\F$. Now we suppose $K$ is $\F$-tight. Then $K_{I\sqcup J}$ is $\F$-tight too, and so we only need to consider the case $I\sqcup J=[m]$. By the K\"{u}nneth theorem, the map
    \[
      (j_I*j_J)_*\colon\widetilde{H}_*(K_I*K_J;\F)\to\widetilde{H}_*(K*K;\F)
    \]
    is injective, where $j_I\colon K_I\to K$ denotes the inclusion. Then it suffices to show the composite $(j_I*j_J)\circ\iota_{I,J}$ is null-homotopic.

    Now we may assume $|K|\subset\R^m$ by identifying a simplex $\{i_1,\ldots,i_k\}\in K$ with
    \[
      \{t_1e_{i_1}+\cdots+t_ke_{i_k}\mid t_1+\cdots+t_k=1,\,t_1\ge 0,\,\ldots,t_k\ge 0\}
    \]
    where $e_1,\ldots,e_m$ is the standard basis of $\R^m$. We may assume $|K*K|\subset\R^{2m}$ in the same way. Consider a homotopy $h_t^i\colon\R^{2m}\times[0,1]\to\R^{2m}$ defined by
    \[
      h_t^i(x_1,\ldots,x_m,y_1,\ldots,y_m) =(x_1,\ldots,(1-t)x_i+ty_i,\ldots,x_m,y_1,\ldots,tx_i+(1-t)y_i,\ldots,y_m)
    \]
    for $(x_1,\ldots,x_m,y_1,\ldots,y_m)\in\R^{2m}$. Then $h_t^i$ restricts to a homotopy $h_t^i\colon|K*K|\times[0,1]\to|K*K|$ such that for $i\in I$,
    \[
      (j_I*j_J)\circ\iota_{I,J}=h_0^i\circ(j_I*j_J)\circ\iota_{I,J}\simeq h_1^i\circ(j_I*j_J)\circ \iota_{I,J}=(j_{I-i}*j_{J\cup i})\circ\iota_{I-i,J\cup i}.
    \]
    Thus for $v\in[m]$, $(j_I*j_J)\circ\iota_{I,J}\simeq(j_v*j_{[m]-v})\circ\iota_{v,[m]-v}$. Since $|v*K_{[m]-v}|$ is contractible, we get $(j_I*j_J)\circ\iota_{I,J}\simeq*$, completing the proof.
  \end{proof}

  In this paper, we have been studying a relationship between Golodness and tightness through tight-neighborliness which perfectly works in dimension 3. However, in dimensions $\ge 4$, tight-neighborliness does not work well because it is not equivalent to tightness as mentioned in Section \ref{introduction}. So we consider the following problem, the weak version of which is Problem \ref{problem 1} for manifold triangulations of dimensions $\ge 4$.

  \begin{problem}
    What condition on closed connected $d$-manifold triangulations with $d\ge 4$ guarantees $\F$-Golodness and $\F$-tightness are equivalent?
  \end{problem}

  One approach is to put a topological condition on manifolds. For example, the condition on the Betti number is stated in Section \ref{introduction}. We also have the following theorem, in which manifolds triangulations are not tight-neighborly.

  \begin{theorem}
    \label{d>3}
    Let $M$ be a triangulation of a closed $(d-1)$-connected $2d$-manifold for $d\ge 2$. Then the following are equivalent:
    \begin{enumerate}
      \item $M$ is $\F$-Golod for any field $\F$;

      \item $M$ is $\F$-tight for any field $\F$;

      \item $M$ is $d$-neighborly;

      \item the fat-wedge filtration of $\R Z_M$ is trivial.
    \end{enumerate}
  \end{theorem}

  \begin{proof}
    The implication (1) $\Rightarrow$ (2) holds by Theorem \ref{preGolod} because $M$ is orientable. Suppose $M$ has a minimal non-face $I$ with $|I|\le d+1$. Then $M_I=\partial\Delta(I)$, implying $H_{|I|-2}(M_I;\F)\ne 0$. Since $M$ is $\F$-tight, the natural map $H_{|I|-2}(M_I;\F)\to H_{|I|-2}(M;\F)$ is injective, and since $M$ is $(d-1)$-connected, $\widetilde{H}_*(M;\F)=0$ for $*<d$. Then we get a contradiction, so we obtain the implication (2) $\Rightarrow$ (3). The implication (3) $\Rightarrow$ (4) follows from \cite[Theorem 1.6]{IK3}. The implication (4) $\Rightarrow$ (1) holds by the fact that $K$ is $\F$-Golod over any field $\F$ whenever $Z_K$ is a suspension, as mentioned in Section \ref{introduction}. Therefore, the proof is complete.
  \end{proof}

\end{document}